\documentclass{article}

\usepackage{amsmath,amssymb,amsfonts,amsthm,mathtools}
\usepackage{xcolor}
\usepackage{mathrsfs}
\usepackage{enumitem}
\usepackage[hidelinks]{hyperref}
\usepackage[nameinlink,capitalize]{cleveref}

\theoremstyle{plain}
\newtheorem{theorem}{Theorem}[section]
\newtheorem{proposition}[theorem]{Proposition}
\newtheorem{lemma}[theorem]{Lemma}
\newtheorem{corollary}[theorem]{Corollary}
\theoremstyle{definition}
\newtheorem{definition}[theorem]{Definition}
\newtheorem{example}[theorem]{Example}
\newtheorem{remark}[theorem]{Remark}

\newcommand{\R}{\mathbb{R}}

\newcommand{\cl}{\operatorname{cl}}
\newcommand{\rank}{\operatorname{rank}}

\newcommand{\Ker}{\operatorname{Ker}}

\newcommand{\cL}{\mathcal{L}}

\newcommand{\oT}{T^{2,\mathrm{out}}}
\newcommand{\iT}{T^{2,\mathrm{in}}}
\newcommand{\aT}{T^{2,\mathrm{arc}}}

\hypersetup{pdftitle={Real radical branch recovery for second order semialgebraic optimization},
pdfauthor={Le Cong Trinh},
pdfsubject={Second order semialgebraic optimization},
pdfkeywords={semialgebraic optimization, parabolic second order tangent set, real radical, nonlinear programming}}

\begin{document}

\title{Real radical branch recovery for second order semialgebraic optimization}

\author{Le Cong Trinh\thanks{Corresponding author: lecongtrinh@qnu.edu.vn}\\
\small Department of Mathematics and Statistics, Quy Nhon University\\
\small Quy Nhon Nam Ward, Gia Lai, Vietnam}
\date{}
\maketitle

\begin{abstract}
We study parabolic second order tangent sets in semialgebraic optimization when the supplied polynomial representation is singular. Our focus is the representation problem: formal differentiation of singular defining equations may lose the intrinsic second order geometry. The main result is a real radical branch recovery theorem. If the local real minimal prime branches of the equality ideal are smooth, then the branches selected by a direction and their exact parabolic second order tangent sets can be recovered from a local real radical decomposition and first- and second order polynomial data, without a preassigned stratification. The formula is invariant under changes of regular branch generators, while counterexamples delimit the roles of real radicalization, branch regularity, and constraint qualification. In optimization, the recovered models reduce each directional second order necessary test to finitely many linear programs in the parabolic correction and hence to branchwise multiplier certificates. A parameterized example shows that the formal second order system of the original singular representation may falsely reject a strict local minimizer, whereas branch recovery gives the correct test. Intrinsically singular branches are separated from representation singularities and can be incorporated when an exact higher order model is available.
\end{abstract}



\section{Introduction}

Tangent constructions are fundamental in variational analysis, real algebraic geometry, and constrained optimization. For smooth manifolds the tangent space gives the first order local model, while for singular or constrained sets one works with tangent cones and related limiting objects; see Rockafellar--Wets \cite{RockafellarWets1998} and Mordukhovich \cite{Mordukhovich2006,Mordukhovich2018}. In the semialgebraic category, curve selection, quantifier elimination, and finite stratifications make tangent geometry especially tractable; we refer to Bochnak--Coste--Roy \cite{BochnakCosteRoy1998}, Coste \cite{Coste2000}, van den Dries \cite{vanDenDries1998}, Lojasiewicz \cite{Lojasiewicz1965}, Whitney \cite{Whitney1965}, and Gibson et al. \cite{Gibson1979}. However, first order tangent cones do not capture curvature, second order activation of constraints, or branchwise feasible behavior. These features are decisive in degenerate polynomial optimization problems.

The object studied in this paper is the parabolic second order tangent set. Given $S\subset\R^n$, $\bar x\in S$, and $u\in T_S(\bar x)$, it consists of the second order corrections $w$ that can occur in feasible expansions
\[
        x(t)=\bar x+t u+\frac12t^2w+o(t^2),\qquad t\downarrow0.
\]
Second order tangent constructions appear in the early work of Ben-Tal \cite{BenTal1980} and Ben-Tal--Zowe \cite{BenTalZowe1982}; they were used by Bonnans, Cominetti, and Shapiro \cite{BCS1999} to formulate second order regularity conditions yielding no-gap second order optimality conditions. Related developments in perturbation analysis, semidefinite programming, epi differentiability, prox regularity, and parabolic regularity can be found in Bonnans--Shapiro \cite{BonnansShapiro2000}, Shapiro \cite{Shapiro1997}, Rockafellar \cite{Rockafellar1988,Rockafellar1989}, Ioffe \cite{Ioffe1991}, Poliquin--Rockafellar \cite{PoliquinRockafellar1996}, Mohammadi--Sarabi \cite{MohammadiSarabi2020}, and Mohammadi--Mordukhovich--Sarabi \cite{MMS2021}. The aim of the present paper is more limited: to compare the intrinsic parabolic tangent set with the set obtained by formally differentiating a chosen polynomial representation, and to isolate precisely where singular representations lose branch information. The general second order optimality theory and the finite union identity used below serve as background for the representation question considered here.

Let
\[
        S=\{x\in\R^n:g_i(x)=0,
        \ i=1,\ldots,r,
        \ h_j(x)\leq0,
        \ j=1,\ldots,s\}
\]
be a basic closed semialgebraic set. A Taylor expansion along a feasible parabolic arc gives equality constraints
\[
        \nabla g_i(\bar x)u=0,
        \qquad
        \nabla g_i(\bar x)w+u^T\nabla^2g_i(\bar x)u=0,
\]
and, for active inequalities with zero first order variation,
\[
        \nabla h_j(\bar x)w+u^T\nabla^2h_j(\bar x)u\leq0.
\]
These conditions define an algebraic second order linearized set $\cL_S^2(\bar x;u)$. The inclusion $T_S^2(\bar x;u)\subseteq \cL_S^2(\bar x;u)$ is automatic, but the reverse inclusion is a genuine geometric issue: one must know whether every formally admissible second order correction is realized by an actual semialgebraic feasible arc.

The general variational theory of second order tangent sets and second order optimality conditions is well developed; see in particular \cite{BCS1998,BCS1999,BonnansShapiro2000,FrankowskaOsmolovskii2024,GfrererYeZhou2022,BenkoEtAl2023,OuyangYeZhang2025,MohammadiSarabi2025,ChenEtAl2026,Ramos2026}. The 1998--1999 papers of Bonnans--Cominetti--Shapiro already develop second order regularity and parabolic second order tangent sets, while recent work treats nonconvex set constraints, disjunctive systems, and asymptotic second order tangent constructions. The issue considered here is instead representation sensitive: we distinguish the intrinsic geometric parabolic tangent set from the formal second order Taylor system attached to a polynomial presentation and give direct realizability proofs under smooth constraint qualifications. The main additional result is an algebraic branch recovery criterion for singular representations: after local real radical decomposition, smooth real minimal prime branches can be selected by first order polynomial data and their exact parabolic tangent sets computed from branch Jacobians and Hessians. This avoids assuming a stratification as input. A finite stratification identity provides the geometric envelope for this construction.

Our optimization setting is finite dimensional nonlinear programming with semialgebraic, in particular polynomial, constraints.  The branch recovery formula has an optimization consequence that is both finite and verifiable: for each critical direction it converts the intrinsic second order test into finitely many linear programs in the parabolic correction, and linear programming duality yields branchwise multiplier certificates.  A parameterized example shows that the second order Taylor system of the original singular representation can falsely reject a strict local minimizer, whereas the recovered branch model gives the correct positive second order test. For
\[
        \min f(x)\quad \text{subject to}\quad x\in S,
        \qquad f\in\R[x],
\]
a feasible parabolic arc gives the second order expression
\[
        Q_f(\bar x;u,w)
        :=\nabla f(\bar x)w+u^T\nabla^2f(\bar x)u.
\]
The optimization consequences of such tangent objects are classical, and recent work treats substantially more general nonconvex set constrained problems \cite{GfrererYeZhou2022,BenkoEtAl2023,OuyangYeZhang2025}. We therefore use optimization only to illustrate why the branchwise computation matters; the standard second order necessary and sufficient conditions are used only as variational tools. In particular, we explicitly separate ordinary parabolic corrections from fractional order semialgebraic arcs, a distinction that is essential at cuspidal branches.

The paper is organized as follows. Section~\ref{sec:preliminaries} recalls the basic tangent and semialgebraic tools. Section~\ref{sec:main-results} states the algebraic formulas and regularity criteria. Section~\ref{sec:proofs} gives proofs. Section~\ref{sec:examples} contains examples and counterexamples. Section~\ref{sec:applications} connects the branchwise formulas with second order conditions in semialgebraic nonlinear programming. A brief conclusion summarizes the main results.

\section{Preliminaries}\label{sec:preliminaries}

\subsection{Tangent cones and parabolic tangent sets}

We recall the first- and second order tangent objects used throughout the paper. The Bouligand tangent cone and its variants are standard in variational analysis; see Rockafellar--Wets \cite[Chapters~6 and 13]{RockafellarWets1998} and Mordukhovich \cite{Mordukhovich2006,Mordukhovich2018}. Second order tangent sets and parabolic tangent constructions are used in second order optimality theory; see Ben-Tal \cite{BenTal1980}, Ben-Tal--Zowe \cite{BenTalZowe1982}, Bonnans--Cominetti--Shapiro \cite{BCS1999}, and Bonnans--Shapiro \cite{BonnansShapiro2000}. Parabolic regularity in the sense of geometric variational analysis was developed systematically by Mohammadi--Mordukhovich--Sarabi \cite{MMS2021}. Related second order variational tools include epi-derivatives and prox regularity, as developed for instance by Rockafellar \cite{Rockafellar1988,Rockafellar1989}, Ioffe \cite{Ioffe1991}, and Poliquin--Rockafellar \cite{PoliquinRockafellar1996}.

Let $S\subset\R^n$ and let $\bar x\in S$. The Bouligand tangent cone of $S$ at $\bar x$ is
\[
T_S(\bar x):=
\left\{u\in\R^n:
\exists t_k\downarrow0,\ \exists u_k\to u
\text{ such that }\bar x+t_k u_k\in S
\right\}.
\]
If $u\in T_S(\bar x)$, the outer parabolic second order tangent set of $S$ at $\bar x$ in direction $u$ is
\[
\oT_S(\bar x;u):=
\left\{w\in\R^n:
\exists t_k\downarrow0,\ \exists w_k\to w
\text{ such that }
\bar x+t_k u+\frac12t_k^2w_k\in S
\right\}.
\]
When no confusion is possible we write $T_S^2(\bar x;u)$ instead of $\oT_S(\bar x;u)$.

The inner parabolic tangent set is
\[
\iT_S(\bar x;u):=
\left\{w\in\R^n:
\forall t_k\downarrow0,\ \exists w_k\to w
\text{ such that }
\bar x+t_k u+\frac12t_k^2w_k\in S
\right\}.
\]
Clearly
\[
        \iT_S(\bar x;u)\subseteq \oT_S(\bar x;u).
\]
Following the variational terminology of \cite{BCS1999,MMS2021}, we say that $S$ is parabolically regular at $\bar x$ in direction $u$ if
\[
        \iT_S(\bar x;u)=\oT_S(\bar x;u).
\]

For semialgebraic sets it is useful to introduce an arc-generated variant. We define
\[
\aT_S(\bar x;u):=
\left\{w\in\R^n:
\begin{array}{l}
\exists\varepsilon>0,\ \exists\gamma:(0,\varepsilon)\to S\text{ semialgebraic such that}\\[1mm]
\gamma(t)=\bar x+t u+\dfrac12t^2w+o(t^2)
\end{array}
\right\}.
\]
Then
\[
        \aT_S(\bar x;u)\subseteq\iT_S(\bar x;u)\subseteq\oT_S(\bar x;u).
\]
Indeed, a single realizing arc works along every prescribed sequence $t_k\downarrow0$ after absorbing its $o(t^2)$ remainder into $w_k$.
For semialgebraic sets the curve selection lemma converts many limiting constructions into semialgebraic arcs; see \cite[Section~2.5]{BochnakCosteRoy1998}, \cite{Coste2000}, and \cite[Chapter~3]{vanDenDries1998}. Preserving a prescribed second order coefficient is, however, a stronger property. This is the main reason why parabolic regularity is nontrivial in the present algebraic setting.

\subsection{Semialgebraic sets and active constraints}

We use standard terminology from real algebraic and semialgebraic geometry. A subset $S\subset\R^n$ is semialgebraic if it is a finite union of sets defined by finitely many polynomial equalities and inequalities. The basic structural facts used below are the Tarski--Seidenberg principle, the semialgebraic curve selection lemma, and the existence of finite $C^p$ semialgebraic stratifications; see Bochnak--Coste--Roy \cite{BochnakCosteRoy1998}, Coste \cite{Coste2000}, van den Dries \cite{vanDenDries1998}, and the foundational semianalytic notes of Lojasiewicz \cite{Lojasiewicz1965}. The use of strata and tangent limits is also connected with the classical work of Whitney \cite{Whitney1965} and with the singularity-theoretic treatment of stratified mappings in Gibson--Wirthmueller--du Plessis--Looijenga \cite{Gibson1979}.

In most of the paper we consider basic closed semialgebraic sets
\[
        S=\{x\in\R^n:g_i(x)=0,\ i=1,\ldots,r,\ h_j(x)\leq0,\ j=1,\ldots,s\},
\]
where $g_i,h_j\in\R[x]$. We write
\[
        g=(g_1,\ldots,g_r),\qquad h=(h_1,\ldots,h_s).
\]
The active inequality index set at $\bar x\in S$ is
\[
        I(\bar x):=\{j\in\{1,\ldots,s\}:h_j(\bar x)=0\}.
\]
For $u\in T_S(\bar x)$, the second order active index set is
\[
        I_0(\bar x;u):=
        \{j\in I(\bar x):\nabla h_j(\bar x)u=0\}.
\]
The inequalities in $I(\bar x)\setminus I_0(\bar x;u)$ have negative first order variation along $u$ and hence impose no second order restriction.

For a polynomial mapping $G=(G_1,\ldots,G_m):\R^n\to\R^m$, we write $DG(x)$ for its Jacobian matrix and
\[
        D^2G(x)[u,u]
        :=\big(u^T\nabla^2G_1(x)u,\ldots,u^T\nabla^2G_m(x)u\big).
\]
The algebraic second order linearized set associated with $S$ at $(\bar x,u)$ is
\[
\cL_S^2(\bar x;u):=
\left\{w\in\R^n:
\begin{array}{l}
Dg(\bar x)w+D^2g(\bar x)[u,u]=0,\\[1mm]
\nabla h_j(\bar x)w+u^T\nabla^2h_j(\bar x)u\leq0,
\quad j\in I_0(\bar x;u)
\end{array}
\right\}.
\]
This is an affine polyhedron in the variable $w$, although its coefficients depend algebraically on $\bar x$ and $u$. In the optimization literature, analogous second order linearized systems appear in second order optimality and perturbation analysis for nonlinear programming and semidefinite programming; see \cite{BCS1999,BonnansShapiro2000,Shapiro1997}.

\subsection{Exact arc realization}\label{subsec:arc-realization}

For later reference, we isolate the genuinely geometric condition behind equality with the Taylor model. We say that $S$ is \emph{second order arc-realizable} at $(\bar x,u)$ if every $w\in\cL_S^2(\bar x;u)$ is realized by a semialgebraic arc
\[
 \gamma(t)=\bar x+tu+\tfrac12t^2w+o(t^2),\qquad \gamma(t)\in S.
\]
By the preceding inclusions and the Taylor inclusion proved below, this condition is equivalent to
\[
 \aT_S(\bar x;u)=\iT_S(\bar x;u)=\oT_S(\bar x;u)=\cL_S^2(\bar x;u).
\]
Thus arc realizability is a useful property to verify in concrete geometric classes, but it is not presented as an independent ``criterion'': once it is assumed, the equality is immediate.

\section{Main results}\label{sec:main-results}

\subsection{The basic inclusion}

The first result records the universal inclusion obtained from the second order Taylor expansion.
\begin{theorem}\label{thm:basic-inclusion}
Let
\[
        S=\{x\in\R^n:g(x)=0,\ h(x)\leq0\}
\]
be a basic closed semialgebraic set. Let $\bar x\in S$ and $u\in T_S(\bar x)$. Then
\[
        \oT_S(\bar x;u)\subseteq \cL_S^2(\bar x;u).
\]
Moreover, $\oT_S(\bar x;u)$ is semialgebraic.
\end{theorem}

Theorem~\ref{thm:basic-inclusion} is one-sided. Equality requires an actual realization statement and cannot be inferred from Taylor expansion alone. In the smooth classes below we prove this realization directly. At singular points, Section~\ref{sec:stratified} instead resolves the set into selected smooth branches.

\subsection{Smooth algebraic models}

We begin the exact realization results with a regular hypersurface.
\begin{theorem}\label{thm:hypersurface}
Let
\[
        S=\{x\in\R^n:p(x)=0\}
\]
near $\bar x$, where $p\in\R[x]$ and $\nabla p(\bar x)\neq0$. Then for every $u\in T_S(\bar x)$,
\[
        T_S^2(\bar x;u)=
        \{w\in\R^n:\nabla p(\bar x)w+u^T\nabla^2p(\bar x)u=0\}.
\]
Thus $S$ is parabolically regular at $\bar x$ in every tangent direction.
\end{theorem}

The same argument extends to regular complete intersections.
\begin{theorem}\label{thm:complete-intersection}
Let
\[
        S=\{x\in\R^n:G(x)=0\},\qquad G=(G_1,\ldots,G_r),
\]
where $G_i\in\R[x]$. Assume that $\rank DG(\bar x)=r$. Then, for every $u\in T_S(\bar x)=\Ker DG(\bar x)$,
\[
        T_S^2(\bar x;u)=
        \{w\in\R^n:DG(\bar x)w+D^2G(\bar x)[u,u]=0\}.
\]
Consequently, regular real algebraic complete intersections are parabolically regular in every tangent direction.
\end{theorem}

We next include smooth inequality constraints and isolate the active constraints that remain critical along the chosen direction.
\begin{theorem}\label{thm:inequality-system}
Let
\[
        S=\{x\in\R^n:g(x)=0,\ h(x)\leq0\},
\]
let $\bar x\in S$, and let $u\in T_S(\bar x)$. Suppose that the gradients
\[
        \nabla g_i(\bar x),\quad i=1,\ldots,r,
        \qquad
        \nabla h_j(\bar x),\quad j\in I_0(\bar x;u),
\]
are linearly independent. Then
\[
        T_S^2(\bar x;u)=\cL_S^2(\bar x;u).
\]
\end{theorem}

\subsection{A stratified parabolic tangent formula}\label{sec:stratified}

The next result explains how second order tangent data of a singular
semialgebraic set decomposes into branchwise contributions.  We recall that
every semialgebraic set admits a finite stratification by Nash manifolds; see
\cite[Theorem~9.1.8]{BochnakCosteRoy1998}.  The proof below uses only the
finiteness and disjointness of the chosen stratification, together with the
sequential definitions of the Bouligand and parabolic second order tangent
sets.  For the latter constructions and their role in second order analysis,
see \cite{BCS1999,RockafellarWets1998,MMS2021}.

To formulate the finite decomposition precisely, we introduce the strata that are visible from a prescribed tangent direction.
\begin{definition}\label{def:selected-strata}
Let $S\subset\R^n$ be a closed semialgebraic set and let
\[
        S=\bigsqcup_{\alpha\in A}M_\alpha
\]
be a finite $C^2$ semialgebraic stratification. For $\bar x\in S$ and $u\in T_S(\bar x)$, define the set of strata selected by $u$ as
\[
        A(\bar x,u):=
        \{\alpha\in A:\bar x\in\cl M_\alpha\text{ and }u\in T_{\cl M_\alpha}(\bar x)\}.
\]
For $\alpha\in A(\bar x,u)$, define the branchwise parabolic tangent set
\[
        T_{\alpha}^2S(\bar x;u):=T^2_{\cl M_\alpha}(\bar x;u).
\]
We also define $T_{\alpha}^{2,\mathrm{arc}}S(\bar x;u)$ as the set of all $w\in\R^n$ for which there exists a semialgebraic arc $\gamma:(0,\varepsilon)\to M_\alpha$ such that
\[
        \gamma(t)=\bar x+t u+\frac12t^2w+o(t^2).
\]
\end{definition}

With this notation, finiteness of the stratification yields the following decomposition.
\begin{proposition}\label{thm:stratified-formula}
Let $S\subset\R^n$ be a closed semialgebraic set and let
\[
S=\bigsqcup_{\alpha\in A}M_\alpha
\]
be any finite $C^2$ semialgebraic stratification. For $\bar x\in S$ and
$u\in T_S(\bar x)$, let $A(\bar x,u)$ be as in Definition~\ref{def:selected-strata}.
Then $A(\bar x,u)$ is finite and nonempty and, without any additional
arc realizability assumption,
\[
T_S^2(\bar x;u)=\bigcup_{\alpha\in A(\bar x,u)
T^2_{\cl M_\alpha}(\bar x;u).}
\]
Thus the parabolic tangent geometry of a closed semialgebraic feasible set is
exactly the finite envelope of the corresponding geometries of the closures
of the strata selected by the first order direction $u$.
\end{proposition}

\begin{remark}\label{rem:decomp-novelty}
This finite decomposition uses only the finiteness of the stratification and the closedness of $S$.  For a single branch it may reduce to an identity.  Its useful role here is computational and representation theoretic:
at a singular point it replaces a degenerate polynomial description by
finitely many branch closures, on which smooth second order formulas may be
applied branch by branch. In particular, no additional branchwise arc regularity assumption is needed: the reverse inclusion follows directly from the set inclusion and the finiteness of the decomposition.
\end{remark}

\subsection{Recovering smooth branches from the real radical ideal}\label{sec:real radical}

The preceding decomposition is geometric: it presupposes a stratification.  For polynomially described feasible sets one would instead like to recover the relevant branches directly from algebraic data.  The next result gives such a criterion for a broad and practically important singular class: the original equality representation may be singular, while the real irreducible branches through the reference point are smooth.  The criterion involves only a local real radical decomposition, Jacobian ranks, and the active inequalities.

We first justify the passage from a localized algebraic decomposition to the corresponding real set germ. For an ideal $J\subset\mathbb R[x_1,\ldots,x_n]$, we write
\[
 V_{\mathbb R}(J):=\{x\in\mathbb R^n:p(x)=0\text{ for every }p\in J\}
\]
for its real zero set, and
\[
 \mathfrak m_{\bar x}:=\langle x_1-\bar x_1,\ldots,x_n-\bar x_n\rangle
\]
for the maximal ideal associated with $\bar x$. If $P_1,\ldots,P_N$ occur below, they are prime ideals in $\mathbb R[x_1,\ldots,x_n]$ whose localizations at $\mathfrak m_{\bar x}$ are the minimal prime components of the localized real radical. Thus every $P_\nu$ contains the equality ideal after passage to the corresponding local germ and represents a real algebraic branch through $\bar x$.

\begin{lemma}\label{lem:local-real-germ}
Let $I\subset\mathbb R[x_1,\ldots,x_n]$ and let $\bar x\in V_{\mathbb R}(I)$. Suppose that prime ideals $P_1,\ldots,P_N$ satisfy
\[
 (\sqrt[\mathbb R]{I})_{\mathfrak m_{\bar x}}
   =\left(\bigcap_{\nu=1}^N P_\nu\right)_{\mathfrak m_{\bar x}}
   =\bigcap_{\nu=1}^N(P_\nu)_{\mathfrak m_{\bar x}}.
\]
Then there is a Euclidean neighborhood $U$ of $\bar x$ such that
\[
 V_{\mathbb R}(I)\cap U
 =\bigcup_{\nu=1}^N\bigl(V_{\mathbb R}(P_\nu)\cap U\bigr).
\]
Consequently the localized decomposition determines the germ of the real zero set at $\bar x$.
\end{lemma}

\begin{proof}
Set $R:=\sqrt[\mathbb R]{I}$ and $Q:=\bigcap_{\nu=1}^N P_\nu$. By hypothesis, $R_{\mathfrak m_{\bar x}}=Q_{\mathfrak m_{\bar x}}$. Since $\mathbb R[x_1,\ldots,x_n]$ is Noetherian, choose finite generating sets $R=\langle r_1,\ldots,r_a\rangle$ and $Q=\langle q_1,\ldots,q_b\rangle$. For every $i$, the inclusion $r_i/1\in Q_{\mathfrak m_{\bar x}}$ gives a polynomial $a_i\notin\mathfrak m_{\bar x}$ such that $a_i r_i\in Q$. Similarly, for every $j$ there is $b_j\notin\mathfrak m_{\bar x}$ such that $b_j q_j\in R$. With
\[
 d:=\prod_{i=1}^a a_i\prod_{j=1}^b b_j,
\]
we have $d(\bar x)\ne0$ and $R_d=Q_d$. Let
\[
 U:=\{x\in\mathbb R^n:d(x)\ne0\},
\]
which is a Euclidean neighborhood of $\bar x$.

We claim that $V_{\mathbb R}(R)\cap U=V_{\mathbb R}(Q)\cap U$. Indeed, if $x\in U$ annihilates $R$ and $q\in Q$, the equality $R_d=Q_d$ yields an integer $k\ge0$ such that $d^kq\in R$; hence $d(x)^kq(x)=0$, and therefore $q(x)=0$. This proves $V_{\mathbb R}(R)\cap U\subseteq V_{\mathbb R}(Q)\cap U$. The reverse inclusion follows in the same way by interchanging $R$ and $Q$.

By the real Nullstellensatz, $V_{\mathbb R}(I)=V_{\mathbb R}(\sqrt[\mathbb R]{I})=V_{\mathbb R}(R)$. Finally, for finitely many ideals,
\[
 V_{\mathbb R}\left(\bigcap_{\nu=1}^N P_\nu\right)
 =\bigcup_{\nu=1}^N V_{\mathbb R}(P_\nu).
\]
Combining these identities on $U$ gives the asserted equality of germs.
\end{proof}

We can now state the main branch recovery result.
\begin{theorem}\label{thm:real radical-branch}
Let
\[
 S=\{x\in\mathbb R^n:g_1(x)=\cdots=g_r(x)=0,\ h_1(x)\le0,\ldots,h_s(x)\le0\}
\]
be closed and semialgebraic, let $I=\langle g_1,\ldots,g_r\rangle$, and let $\bar x\in S$.  Suppose that, after localization at the maximal ideal $\mathfrak m_{\bar x}$, the real radical has finitely many real minimal primes through $\bar x$,
\[
 (\sqrt[\mathbb R]{I})_{\mathfrak m_{\bar x}}
   =\bigcap_{\nu=1}^N (P_\nu)_{\mathfrak m_{\bar x}}.
\]
Put $X_\nu=V_{\mathbb R}(P_\nu)$.  Only the germs at $\bar x$ are used below. Assume that every branch germ $X_\nu$ is regular at $\bar x$ in the following explicit algebraic sense: for each $\nu$ there are polynomials
$G_\nu=(G_{\nu1},\ldots,G_{\nu c_\nu})\subset P_\nu$ defining $X_\nu$ as a germ at $\bar x$ such that
\[
 \operatorname{rank}DG_\nu(\bar x)=c_\nu=\operatorname{codim}_{\bar x}X_\nu.
\]
In particular, each recovered real branch is a $C^2$ manifold germ at $\bar x$.
For a direction $u$, set
\[
 J_\nu(\bar x,u)=\{j:h_j(\bar x)=0,\ \nabla h_j(\bar x)u=0\}.
\]
Assume, for every branch $\nu$ for which
\[
 DG_\nu(\bar x)u=0,\qquad
 \nabla h_j(\bar x)u\le0\quad(h_j(\bar x)=0),                 \tag{RB1}
\]
the branchwise LICQ
\[
 \{\nabla G_{\nu i}(\bar x):1\le i\le c_\nu\}
 \cup\{\nabla h_j(\bar x):j\in J_\nu(\bar x,u)\}
 \quad\hbox{is linearly independent}.                         \tag{RB2}
\]
Then (RB1) is a checkable criterion for whether $u$ selects the feasible branch
$S_\nu:=X_\nu\cap\{h\le0\}$, and
\[
 T_S^2(\bar x;u)=\bigcup_{\nu\in\mathcal B(\bar x,u)}
 \mathcal L_{\nu}^2(\bar x;u),
\]
where $\mathcal B(\bar x,u)$ is the set of indices satisfying (RB1), and
\[
\begin{split}
 \mathcal L_{\nu}^2(\bar x;u)=\{w:\;&DG_\nu(\bar x)w+D^2G_\nu(\bar x)[u,u]=0,\\
 &\nabla h_j(\bar x)w+u^T\nabla^2h_j(\bar x)u\le0,
 \quad j\in J_\nu(\bar x,u)\}.
\end{split}                                                     \tag{RB3}
\]
The formula is determined by the local real radical branch germs and their second order data rather than by the possibly singular original generators $g_i$. Independence from the choice of regular local branch equations is established below in Proposition~\ref{prop:generator-invariance}.
\end{theorem}

\begin{proof}
By Lemma~\ref{lem:local-real-germ}, there is a neighborhood $U$ of $\bar x$ such that
\[
 S\cap U=\bigcup_{\nu=1}^N(S_\nu\cap U),
 \qquad S_\nu:=X_\nu\cap\{h_1\le0,\ldots,h_s\le0\}.
\]
Both the Bouligand tangent cone and the outer parabolic tangent set are local constructions. Since the union is finite, a subsequence argument gives
\[
 T_S^2(\bar x;u)
 =\bigcup_{\nu:\,u\in T_{S_\nu}(\bar x)}T_{S_\nu}^2(\bar x;u).       \tag{4.1}
\]
To see the nontrivial inclusion in this identity, take $w\in T_S^2(\bar x;u)$ and a realizing sequence
$\bar x+t_ku+\frac12t_k^2w_k\in S$, with $t_k\downarrow0$ and $w_k\to w$. For all sufficiently large $k$ the points lie in $U$, and finiteness permits passage to a subsequence contained in one $S_\nu$. The converse inclusion follows from $S_\nu\cap U\subset S$ and locality.

Fix a branch $\nu$. The rank assumption on $DG_\nu(\bar x)$ implies, after shrinking $U$ if necessary, that $X_\nu$ is the $C^2$ manifold germ defined by $G_\nu=0$. Under (RB2), the usual tangent formula for a smooth equality and inequality system gives
\[
 u\in T_{S_\nu}(\bar x)
 \quad\Longleftrightarrow\quad
 DG_\nu(\bar x)u=0,
 \qquad
 \nabla h_j(\bar x)u\le0
 \quad\text{for every }j\text{ with }h_j(\bar x)=0.
\]
Thus the branches that contribute to (4.1) are exactly those satisfying (RB1).

For such a branch, Theorem~\ref{thm:inequality-system} applies to the local system $G_\nu=0$, $h\le0$. The inequalities with negative first order variation are automatically strict along a realizing arc for small positive parameter, whereas the indices in $J_\nu(\bar x,u)$ are precisely those that remain relevant at second order. Hence
\[
 T_{S_\nu}^2(\bar x;u)=\mathcal L_\nu^2(\bar x;u).
\]
Substituting this equality into (4.1) proves the formula.
\end{proof}

\begin{remark}\label{rem:branch-checkable}
The hypotheses are algebraic and finite.  Algorithms for real radical computation followed by minimal prime decomposition identify the relevant local $P_\nu$ in principle; Jacobian minors test regularity and (RB2); (RB1) and (RB3) use only gradients and Hessians.  Thus Theorem~\ref{thm:real radical-branch} removes the need to know a Whitney or Nash stratification in advance.  It does not cover branches singular in their own right (for example a cusp); there fractional or higher order data can be unavoidable.  The point is precisely to separate \emph{singularity of the chosen representation} from \emph{intrinsic singularity of each real branch}.
\end{remark}

When no inequalities are present, the formula simplifies immediately.
\begin{corollary}\label{cor:equality-branch}
If $s=0$, then for every $u\in T_S(\bar x)$,
\[
 T_S^2(\bar x;u)=
 \bigcup_{\substack{1\le\nu\le N\\DG_\nu(\bar x)u=0}}
 \{w:DG_\nu(\bar x)w+D^2G_\nu(\bar x)[u,u]=0\}.
\]
Hence a singular equality description can have an exactly computable parabolic tangent set whenever its local real radical branches are smooth.
\end{corollary}

The finite union formula also reduces optimization over the parabolic correction to finitely many branch problems.
\begin{corollary}\label{cor:branchwise-Q}
Under the hypotheses of Theorem~\ref{thm:real radical-branch}, let $f\in C^2$ near $\bar x$ and put
\[
Q_f(\bar x;u,w)=\nabla f(\bar x)w+u^T\nabla^2f(\bar x)u.
\]
For every $u\in T_S(\bar x)$,
\[
\inf_{w\in T_S^2(\bar x;u)}Q_f(\bar x;u,w)
=\min_{\nu\in\mathcal B(\bar x,u)}
\inf_{w\in\mathcal L_\nu^2(\bar x;u)}Q_f(\bar x;u,w),
\]
with $\inf\varnothing=+\infty$. Thus the parabolic test at a singular original representation reduces to finitely many affine-linear optimization problems determined by the real branches.
\end{corollary}

\begin{proof}[Proof of Corollary~\ref{cor:equality-branch}]
Set $s=0$ in Theorem~\ref{thm:real radical-branch}.  Then (RB2) is exactly the full-rank condition on $DG_\nu(\bar x)$ and there are no inequality terms in (RB3).
\end{proof}

\begin{proof}[Proof of Corollary~\ref{cor:branchwise-Q}]
By Theorem~\ref{thm:real radical-branch}, $T_S^2(\bar x;u)$ is the finite union of the sets $\mathcal L_\nu^2(\bar x;u)$ over $\nu\in\mathcal B(\bar x,u)$. The infimum of $Q_f(\bar x;u,\cdot)$ over a finite union is the minimum of the corresponding infima. The convention $\inf\varnothing=+\infty$ covers empty branch models.
\end{proof}

\subsection{Hypotheses and scope of the branch recovery theorem}\label{sec:robustness}

The branch formula is intrinsic and does not depend on the particular regular defining system.
\begin{proposition}\label{prop:generator-invariance}
Let $X$ be a $C^2$ algebraic manifold germ at $\bar x$, and let $G$ and $\widetilde G$ be two polynomial systems that define the same germ and have full row rank equal to $\operatorname{codim}_{\bar x}X$. Then, for every $u\in T_X(\bar x)$,
\[
 \{w:DG(\bar x)w+D^2G(\bar x)[u,u]=0\}
 =
 \{w:D\widetilde G(\bar x)w+D^2\widetilde G(\bar x)[u,u]=0\}.
\]
Consequently the branch models $\mathcal L_\nu^2(\bar x;u)$ in Theorem~\ref{thm:real radical-branch} do not depend on the chosen regular local generators of $P_\nu$.
\end{proposition}

\begin{proof}
Both Jacobian row spaces equal the conormal space $N_X(\bar x)$. After shrinking the neighborhood, the constant-rank theorem and Hadamard's lemma give smooth matrix functions $A$ and $B$ such that $\widetilde G=A G$ and $G=B\widetilde G$ on a neighborhood of $\bar x$, with $A(\bar x)$ and $B(\bar x)$ inducing inverse maps on the conormal spaces. If $DG(\bar x)u=0$, differentiating $\widetilde G=A G$ twice along $u$ gives
\[
 D\widetilde G(\bar x)w+D^2\widetilde G(\bar x)[u,u]
 =A(\bar x)\big(DG(\bar x)w+D^2G(\bar x)[u,u]\big),
\]
because $G(\bar x)=0$ and $DG(\bar x)u=0$. The reverse implication follows from $G=B\widetilde G$. The inequality part of $\mathcal L_\nu^2$ is unchanged.
\end{proof}

The next examples built into the proposition explain why the principal hypotheses cannot simply be deleted.
\begin{proposition}\label{prop:sharpness}
The three structural ingredients in Theorem~\ref{thm:real radical-branch} address genuinely different obstructions.
\begin{enumerate}
\item Ordinary radicalization is not enough for real local geometry. For $I=\langle x^2+y^2\rangle\subset\mathbb R[x,y]$, one has $\sqrt I=I$, while $V_{\mathbb R}(I)=\{0\}$ and $\sqrt[\mathbb R]{I}=\langle x,y\rangle$. Thus the complex algebraic codimension one prime does not describe the real branch germ; real radicalization recovers the correct zero-dimensional branch.
\item Smoothness of the recovered branch cannot be omitted while retaining a purely quadratic model. For the cusp $X=\{(x,y):y^2=x^3\}$ and $u=(1,0)$, one has $u\in T_X(0)$ but
\[
 T_X^2(0;u)=\varnothing.
\]
Indeed points on a nonzero real branch satisfy $|y|=x^{3/2}$, so an expansion $x=t+O(t^2)$ forces $y/t^2$ to be of order $t^{-1/2}$ and hence no finite parabolic correction exists. Fractional or higher order scaling is therefore intrinsic here.
\item A constraint qualification for second order active inequalities is essential. For $S=\{x\in\mathbb R:x^2\le0\}=\{0\}$ at $\bar x=0$ and $u=0$, the formal second order system based on the vanishing gradient of $h(x)=x^2$ imposes no condition on $w$, whereas $T_S^2(0;0)=\{0\}$. Thus (RB2), or a genuine substitute such as an appropriate metric-subregularity/parabolic derivability condition, cannot be discarded.
\end{enumerate}
\end{proposition}

The finite union mechanism also permits intrinsically singular branches once their exact parabolic models are supplied independently.
\begin{proposition}\label{thm:plugin-singular}
In Theorem~\ref{thm:real radical-branch}, the smoothness assumption is needed only to obtain the explicit Hessian model (RB3); it is not needed for the finite branch decomposition itself. More precisely, suppose the localized real radical decomposition yields closed semialgebraic branch germs $S_\nu$ and, for a fixed $u$, an independently verified exact model $E_\nu(\bar x;u)$ satisfying
\[
 T_{S_\nu}^2(\bar x;u)=E_\nu(\bar x;u)
\quad\text{for every selected branch }\nu.
\]
Then
\[
 T_S^2(\bar x;u)=\bigcup_{\nu:\,u\in T_{S_\nu}(\bar x)}E_\nu(\bar x;u).
\]
Hence intrinsically singular branches may be incorporated whenever their second order tangent sets are available from a separate parabolic derivability, metric-subregularity, $p$-regularity, or higher order analysis.
\end{proposition}

\begin{proof}
The localized real radical decomposition gives $S=\bigcup_\nu S_\nu$ as germs at $\bar x$. The finite union subsequence argument yields
\[
T_S^2(\bar x;u)=\bigcup_{\nu:\,u\in T_{S_\nu}(\bar x)}T_{S_\nu}^2(\bar x;u),
\]
and substitution of the verified models $E_\nu$ proves the assertion. The theorem is deliberately a plug-in principle rather than a claim that every singular branch admits a quadratic Hessian description; Proposition~\ref{prop:sharpness}(2) shows that such a claim would be false. Recent $p$-regularity constructions for degenerate optimization provide one possible source of higher order branch models; see \cite{Ramos2026}.
\end{proof}

\section{Proofs of main results}\label{sec:proofs}

\begin{proof}[Proof of \cref{thm:basic-inclusion}]
Let $w\in\oT_S(\bar x;u)$. Then there exist $t_k\downarrow0$ and $w_k\to w$ such that
\[
        x_k:=\bar x+t_k u+\frac12t_k^2w_k\in S.
\]
Fix an equality constraint $g_i$. Taylor expansion at $\bar x$ gives
\begin{align*}
0=g_i(x_k)
&=g_i(\bar x)+\nabla g_i(\bar x)(x_k-\bar x)
  +\frac12(x_k-\bar x)^T\nabla^2g_i(\bar x)(x_k-\bar x)
  +o(\|x_k-\bar x\|^2)\\
&=t_k\nabla g_i(\bar x)u
+\frac12t_k^2\big(\nabla g_i(\bar x)w_k+u^T\nabla^2g_i(\bar x)u\big)
+o(t_k^2),
\end{align*}
where we used $g_i(\bar x)=0$ and boundedness of $w_k$. Dividing by $t_k$ and passing to the limit yields
\[
        \nabla g_i(\bar x)u=0.
\]
Returning to the expansion, dividing by $t_k^2/2$, and passing to the limit gives
\[
        \nabla g_i(\bar x)w+u^T\nabla^2g_i(\bar x)u=0.
\]

Now let $j\in I_0(\bar x;u)$. Since $h_j(x_k)\leq0$ and $h_j(\bar x)=0$, Taylor expansion gives
\[
0\geq h_j(x_k)
=t_k\nabla h_j(\bar x)u
+\frac12t_k^2\big(\nabla h_j(\bar x)w_k+u^T\nabla^2h_j(\bar x)u\big)
+o(t_k^2).
\]
Because $j\in I_0(\bar x;u)$, the first order term vanishes. Dividing by $t_k^2/2$ and passing to the limit gives
\[
        \nabla h_j(\bar x)w+u^T\nabla^2h_j(\bar x)u\leq0.
\]
Thus $w\in\cL_S^2(\bar x;u)$.

It remains to prove semialgebraicity. The condition $w\in\oT_S(\bar x;u)$ is equivalent to the first order formula over the ordered field of real numbers
\[
\forall \varepsilon>0\ \forall \delta>0\ \exists t\ \exists z
\left(
0<t<\delta,
\ \|z-w\|^2<\varepsilon^2,
\ \bar x+t u+\frac12t^2z\in S
\right).
\]
Since membership in $S$ is described by finitely many polynomial equalities and inequalities, the Tarski-Seidenberg theorem implies that the set of all such $w$ is semialgebraic.
\end{proof}

\begin{proof}[Proof of \cref{thm:hypersurface}]
The inclusion from left to right follows from \cref{thm:basic-inclusion}. We prove the reverse inclusion. Let $w\in\R^n$ satisfy
\[
        \nabla p(\bar x)w+u^T\nabla^2p(\bar x)u=0.
\]
Since $\nabla p(\bar x)\neq0$, after a permutation of coordinates we may write $x=(y,z)\in\R^{n-1}\times\R$ and assume $p_z(\bar x)\neq0$. By the implicit function theorem, locally
\[
        S=\{(y,z):z=\varphi(y)\}
\]
for a real analytic semialgebraic function $\varphi$. Write $\bar x=(\bar y,\bar z)$, $u=(u_y,u_z)$, and $w=(w_y,w_z)$. Since $u\in T_S(\bar x)$,
\[
        u_z=D\varphi(\bar y)u_y.
\]
Define
\[
        y(t)=\bar y+t u_y+\frac12t^2w_y,
        \qquad
        z(t)=\varphi(y(t)).
\]
Then $\gamma(t)=(y(t),z(t))\in S$. Taylor expansion gives
\[
        z(t)=\bar z+tD\varphi(\bar y)u_y
        +\frac12t^2\big(D\varphi(\bar y)w_y+D^2\varphi(\bar y)[u_y,u_y]\big)+o(t^2).
\]
It remains to check that the second order coefficient is $w_z$. Differentiating $p(y,\varphi(y))=0$ once gives
\[
        D_y p(\bar x)+p_z(\bar x)D\varphi(\bar y)=0.
\]
Differentiating twice in direction $u_y$ gives
\[
        u^T\nabla^2p(\bar x)u+p_z(\bar x)D^2\varphi(\bar y)[u_y,u_y]=0.
\]
The assumed identity for $w$ becomes
\[
        D_y p(\bar x)w_y+p_z(\bar x)w_z+u^T\nabla^2p(\bar x)u=0.
\]
Substituting the two previous identities yields
\[
        p_z(\bar x)\big(w_z-D\varphi(\bar y)w_y-D^2\varphi(\bar y)[u_y,u_y]\big)=0.
\]
Since $p_z(\bar x)\neq0$,
\[
        w_z=D\varphi(\bar y)w_y+D^2\varphi(\bar y)[u_y,u_y].
\]
Hence $\gamma(t)=\bar x+t u+\frac12t^2w+o(t^2)$, proving $w\in T_S^2(\bar x;u)$.
\end{proof}

\begin{proof}[Proof of \cref{thm:complete-intersection}]
Since $\rank DG(\bar x)=r$, after a permutation of coordinates we may write $x=(y,z)\in\R^{n-r}\times\R^r$ so that $D_zG(\bar x)$ is invertible. By the implicit function theorem, locally
\[
        S=\{(y,z):z=\Phi(y)\}
\]
for a real analytic semialgebraic map $\Phi$. Let $u=(u_y,u_z)$ and $w=(w_y,w_z)$ satisfy
\[
        DG(\bar x)w+D^2G(\bar x)[u,u]=0.
\]
Define
\[
        y(t)=\bar y+t u_y+\frac12t^2w_y,
        \qquad z(t)=\Phi(y(t)).
\]
Then $\gamma(t)=(y(t),z(t))\in S$. As in the hypersurface case, differentiating the identity $G(y,\Phi(y))=0$ once and twice shows that the condition
\[
        DG(\bar x)w+D^2G(\bar x)[u,u]=0
\]
is equivalent to
\[
        w_z=D\Phi(\bar y)w_y+D^2\Phi(\bar y)[u_y,u_y].
\]
Therefore
\[
        \gamma(t)=\bar x+t u+\frac12t^2w+o(t^2),
\]
so $w\in T_S^2(\bar x;u)$. The reverse inclusion follows from \cref{thm:basic-inclusion}.
\end{proof}

\begin{proof}[Proof of \cref{thm:inequality-system}]
The inclusion $T_S^2(\bar x;u)\subseteq\cL_S^2(\bar x;u)$ follows from Theorem~\ref{thm:basic-inclusion}. For the converse fix $w\in\cL_S^2(\bar x;u)$ and put
\[
 J(w):=\{j\in I_0(\bar x;u):\nabla h_j(\bar x)w+u^T\nabla^2h_j(\bar x)u=0\}.
\]
Because the gradients of the equalities together with all $\nabla h_j(\bar x)$, $j\in I_0(\bar x;u)$, are linearly independent, the smaller family consisting of $g_i$ and $h_j$, $j\in J(w)$, also has full row rank. Hence
\[
 M_w:=\{x:g(x)=0,\ h_j(x)=0\ (j\in J(w))\}
\]
is a $C^2$ semialgebraic manifold near $\bar x$. The standard graph construction used in Theorem~\ref{thm:complete-intersection}, applied to the defining map $(g,(h_j)_{j\in J(w)})$, gives a semialgebraic $C^2$ arc $\gamma(t)\in M_w$ with
\[
 \gamma(t)=\bar x+tu+\tfrac12t^2w+o(t^2).
\]
For $j\in I_0(\bar x;u)\setminus J(w)$ the second order coefficient is strictly negative, so
\[
 h_j(\gamma(t))=\tfrac12t^2\big(\nabla h_j(\bar x)w+u^T\nabla^2h_j(\bar x)u\big)+o(t^2)<0
\]
for all small $t>0$. For $j\in I(\bar x)\setminus I_0(\bar x;u)$ one has $\nabla h_j(\bar x)u<0$, and therefore $h_j(\gamma(t))<0$ for small $t>0$ as well. Inactive inequalities remain inactive by continuity. Thus $\gamma(t)\in S$ and realizes $w$, proving $w\in T_S^2(\bar x;u)$. This also resolves the borderline zero-coefficient case: those inequalities are imposed as equalities in the auxiliary manifold $M_w$, rather than having their sign inferred from an $o(t^2)$ remainder.
\end{proof}

\begin{proof}[Proof of \cref{thm:stratified-formula}]
Finiteness of $A(\bar x,u)$ is immediate. For nonemptiness, choose
$t_k\downarrow0$ and $u_k\to u$ with $x_k=\bar x+t_ku_k\in S$. Since the
stratification is finite, a subsequence lies in one stratum $M_\alpha$.
Thus $\bar x\in\cl M_\alpha$ and $u\in T_{\cl M_\alpha}(\bar x)$, so
$\alpha\in A(\bar x,u)$.

Let $w\in T_S^2(\bar x;u)$ and choose $t_k\downarrow0$, $w_k\to w$ such that
\[
x_k=\bar x+t_ku+\tfrac12t_k^2w_k\in S.
\]
After passage to a subsequence, $x_k\in M_\alpha$ for one $\alpha$. Then
$\bar x\in\cl M_\alpha$ and $(x_k-\bar x)/t_k\to u$, hence
$\alpha\in A(\bar x,u)$. The same sequence gives
$w\in T^2_{\cl M_\alpha}(\bar x;u)$. Therefore
\[
T_S^2(\bar x;u)\subseteq\bigcup_{\alpha\in A(\bar x,u)}
T^2_{\cl M_\alpha}(\bar x;u).
\]
Conversely, $S$ is closed and $M_\alpha\subset S$, so
$\cl M_\alpha\subset S$. Monotonicity of the outer parabolic tangent set
under set inclusion gives
$T^2_{\cl M_\alpha}(\bar x;u)\subseteq T_S^2(\bar x;u)$ for every selected
$\alpha$. Taking the union proves the reverse inclusion and the formula.
\end{proof}

\section{Examples and counterexamples}\label{sec:examples}

The following examples illustrate why parabolic second order tangent sets are more informative than ordinary tangent cones and why exactness of the polynomial second order linearization is a genuine geometric issue.

The tacnode gives a basic illustration in which the singular presentation separates into two smooth real branches.
\begin{example}\label{ex:tacnode}
Let
\[
        S=\{(x,y)\in\R^2:y^2=x^4\}
        =S_+\cup S_-,
        \qquad
        S_\pm=\{(x,y):y=\pm x^2\}.
\]
Here the singular principal ideal has real radical decomposition
\[
 \sqrt[\mathbb R]{\langle y^2-x^4\rangle}
 =\langle y-x^2\rangle\cap\langle y+x^2\rangle.
\]
Both real minimal prime branches are smooth at the origin, so Theorem~\ref{thm:real radical-branch} applies without supplying a stratification. At the origin, both branches have the same tangent direction $u=(1,0)$. The upper branch gives $w_2=2$, while the lower branch gives $w_2=-2$. Therefore
\[
        T_S^2(0;u)=\{w:w_2=2\}\cup\{w:w_2=-2\}.
\]
This example is not used below as a separate singularity application, but it is important for optimization: if such a set occurs as a feasible region, second order tests must be applied branchwise rather than to a single linearized equation.
\end{example}

The cusp shows why ordinary quadratic scaling can fail on an intrinsically singular branch.
\begin{example}\label{ex:cusp}
Let
\[
        S=\{(x,y)\in\R^2:y^2=x^3,
        \ x\geq0\}.
\]
The parametrization $\gamma(s)=(s^2,s^3)$ gives, after writing $t=s^2$,
\[
        \gamma(t)=(t,t^{3/2}).
\]
The deviation from $t(1,0)$ is of order $t^{3/2}$, not $t^2$. Thus the ordinary parabolic scale does not capture the first nontrivial correction of this cusp. In constrained optimization this means that a purely quadratic test may be inconclusive along such a branch; higher order or fractional order information may be required.
\end{example}

A complementarity type set illustrates the finite branch mechanism in a familiar optimization geometry.
\begin{example}\label{ex:complementarity}
Let
\[
        S=\{(x,y)\in\R^2:x\geq0,
        \ y\geq0,
        \ xy=0\}.
\]
At the origin,
\[
        T_S(0)=\R_+(1,0)\cup\R_+(0,1).
\]
For $u=(1,0)$, the selected branch is the positive $x$-axis, and
\[
        T_S^2(0;u)=\{w=(w_1,w_2):w_2=0\}.
\]
For $u=(0,1)$,
\[
        T_S^2(0;u)=\{w=(w_1,w_2):w_1=0\}.
\]
The naive equation $xy=0$ has zero gradient at the origin and gives no second order restriction. The correct parabolic tangent set is obtained only after recognizing the branch structure of the feasible set.
\end{example}
\section{Second order implications for semialgebraic nonlinear programming}\label{sec:applications}

This section records consequences of the preceding tangent formulas in the language of polynomial optimization. The optimality principles themselves are classical; compare \cite{BonnansShapiro2000,GfrererYeZhou2022,BenkoEtAl2023,OuyangYeZhang2025}. The purpose here is to show how a branchwise semialgebraic computation feeds into those principles. Throughout this section,
\[
        \min f(x)\quad\text{subject to}\quad x\in S
\]
where $f\in\R[x]$ and $S\subset\R^n$ is a locally closed basic semialgebraic set. The global theory of polynomial optimization often relies on moment and sums-of-squares relaxations, as initiated by Lasserre \cite{Lasserre2001}, or on algebraic critical varieties such as the gradient ideal method of Nie, Demmel, and Sturmfels \cite{NieDemmelSturmfels2006}. The results below are local. They provide checkable second order tests near a feasible point once the parabolic tangent geometry of $S$ is understood.

A tangent direction $u\in T_S(\bar x)$ is called critical for $f$ at $\bar x$ if
\[
        \nabla f(\bar x)u=0.
\]
For a feasible parabolic arc
\[
        x(t)=\bar x+t u+\frac12t^2w+o(t^2),
\]
Taylor expansion gives
\[
        f(x(t))=f(\bar x)+t\nabla f(\bar x)u
        +\frac12t^2\big(\nabla f(\bar x)w
        +u^T\nabla^2f(\bar x)u\big)+o(t^2).
\]
We write
\[
        Q_f(\bar x;u,w)
        :=\nabla f(\bar x)w+u^T\nabla^2f(\bar x)u.
\]

We recall the intrinsic second order necessary condition in the notation used here.
\begin{theorem}\label{thm:necessary-opt}
If $\bar x$ is a local minimizer of $f$ over $S$, then for every critical direction $u\in T_S(\bar x)$ and every $w\in T_S^2(\bar x;u)$,
\[
        Q_f(\bar x;u,w)\geq0.
\]
If, in addition, $T_S^2(\bar x;u)=\cL_S^2(\bar x;u)$ has been verified by one of the smooth or branchwise formulas above, the same condition may be tested on $\cL_S^2(\bar x;u)$.
\end{theorem}

\begin{proof}
Let $u$ be critical and let $w\in T_S^2(\bar x;u)$. Choose $t_k\downarrow0$ and $w_k\to w$ such that
\[
        x_k=\bar x+t_k u+\frac12t_k^2w_k\in S.
\]
Since $\bar x$ is a local minimizer, $f(x_k)-f(\bar x)\geq0$ for all large $k$. Taylor expansion gives
\[
        f(x_k)-f(\bar x)
        =t_k\nabla f(\bar x)u
        +\frac12t_k^2\big(\nabla f(\bar x)w_k
        +u^T\nabla^2f(\bar x)u\big)
        +o(t_k^2).
\]
The first order term vanishes because $u$ is critical. Dividing by $t_k^2/2$ and passing to the limit yields $Q_f(\bar x;u,w)\geq0$. When exactness $T_S^2(\bar x;u)=\cL_S^2(\bar x;u)$ has been established, the last assertion is immediate.
\end{proof}

Combining the intrinsic necessary condition with branch recovery gives a finite family of exact tests.
\begin{corollary}\label{cor:branch-opt}
Assume the hypotheses of Theorem~\ref{thm:real radical-branch} and let $\bar x$ be a local minimizer of $f$ on $S$.  Then, for every critical direction $u$,
\[
 Q_f(\bar x;u,w)\ge0
 \quad\text{for every }\nu\in\mathcal B(\bar x,u)
 \text{ and every }w\in\mathcal L_\nu^2(\bar x;u).             \tag{SO-BR}
\]
Consequently the second order necessary test for a singular original polynomial representation reduces to finitely many linear feasibility/optimization problems in the correction variable $w$, one for each real radical branch selected by $u$.
\end{corollary}

\begin{proof}
Combine Theorem~\ref{thm:real radical-branch} with Theorem~\ref{thm:necessary-opt}.
\end{proof}

Before applying a constraint qualification on a recovered branch, redundant inequalities may be removed without changing the local geometry.
\begin{proposition}\label{prop:irredundant}
In Theorem~\ref{thm:real radical-branch}, fix a recovered smooth branch $X_\nu$ and delete from the inequality list every constraint that is locally implied as an equality or is locally redundant on $X_\nu$.  If the remaining second order active inequalities satisfy LICQ together with a regular defining system of $X_\nu$, then the formula (RB3) and all branchwise optimization tests remain unchanged.  Consequently, (RB2) is required only for an irredundant branchwise representation, not for the original polynomial list.
\end{proposition}

\begin{proof}
Deleting a constraint that is locally identically zero on $X_\nu$ or locally implied by the remaining branchwise constraints does not change the feasible germ $S_\nu$.  Tangent and parabolic tangent sets are local set invariants, hence are unchanged.  Applying Theorem~\ref{thm:inequality-system} to any irredundant branchwise representation satisfying LICQ gives the same intrinsic set $T^2_{S_\nu}(\bar x;u)$; generator invariance is Proposition~\ref{prop:generator-invariance}.
\end{proof}

Linear programming duality converts each branch test into a multiplier certificate.
\begin{corollary}\label{cor:branch-dual}
Under the hypotheses of Corollary~\ref{cor:branch-opt}, fix a critical direction $u$ and a selected branch $\nu$. Write
\[
 A_\nu=DG_\nu(\bar x),\quad b_\nu=-D^2G_\nu(\bar x)[u,u],
\]
and let $C_\nu w\le d_\nu$ denote the second order active inequalities in (RB3), i.e.
\[
 (C_\nu)_j=\nabla h_j(\bar x),\qquad
 (d_\nu)_j=-u^T\nabla^2h_j(\bar x)u.
\]
Assume $\mathcal L_\nu^2(\bar x;u)\ne\varnothing$. Then the branchwise necessary condition
\[
Q_f(\bar x;u,w)\ge0\qquad\forall w\in\mathcal L_\nu^2(\bar x;u)
\]
is equivalent to the existence of $\lambda_\nu$ and $\mu_\nu\ge0$ such that
\[
 \nabla f(\bar x)+A_\nu^T\lambda_\nu+C_\nu^T\mu_\nu=0,                 \tag{BR-KKT1}
\]
\[
 u^T\nabla^2f(\bar x)u-\lambda_\nu^Tb_\nu-\mu_\nu^Td_\nu\ge0.       \tag{BR-KKT2}
\]
Equivalently,
\[
 u^T\nabla^2_{xx}\!\Big(f+\lambda_\nu^TG_\nu+\sum_{j\in J_\nu}\mu_{\nu j}h_j\Big)(\bar x)u\ge0. \tag{BR-LAG}
\]
Thus a singular polynomial representation gives a finite family of classical in form second order multiplier certificates, one for each real branch selected by $u$.
\end{corollary}

\begin{proof}
On $\mathcal L_\nu^2$, minimizing $Q_f$ is the linear program
\[
\min_w\{\nabla f(\bar x)^Tw+u^T\nabla^2f(\bar x)u:A_\nu w=b_\nu,\ C_\nu w\le d_\nu\}.
\]
Since the feasible set is nonempty and the desired assertion is precisely that its optimal value is nonnegative (including the case of an unbounded feasible polyhedron with finite lower bound), linear programming duality/Farkas' lemma gives (BR-KKT1)--(BR-KKT2). Substituting $b_\nu=-D^2G_\nu[u,u]$ and $d_{\nu j}=-u^T\nabla^2h_j u$ yields (BR-LAG).
\end{proof}

The following optimization problem shows that using the original singular representation can produce a genuinely incorrect second order conclusion.
\begin{example}\label{ex:false-rejection}
Consider, for a parameter $c\in(0,1)$,
\[
 \min f_c(x,y):=y-cx^2
 \quad\text{subject to}\quad
 y^2-x^4=0,\qquad y\ge0.                                      \tag{EX}
\]
The feasible germ is the single smooth parabola $y=x^2$, so $(0,0)$ is in fact a strict local minimizer because
$f_c(x,x^2)=(1-c)x^2>0$ for $x\ne0$.  Nevertheless, if one uses the singular equality generator
$p(x,y)=y^2-x^4$ directly, then for $u=(1,0)$ both $\nabla p(0)$ and
$u^T\nabla^2p(0)u$ vanish.  The inequality $-y\le0$ only yields $w_2\ge0$ in the formal second order Taylor system.  Hence the spurious correction $w=(0,0)$ is formally admissible and
\[
 Q_{f_c}(0;u,0)=-2c<0.
\]
Thus replacing the true parabolic tangent set by the second order linearization of the supplied singular representation would incorrectly reject a genuine strict local minimizer.

The real radical preprocessing underlying branch recovery removes this pathology.  The equality ideal factors into the two real branches $y-x^2=0$ and $y+x^2=0$; restriction of the inequality $y\ge0$ to these branches shows that the lower branch contributes only the origin, whereas the upper branch is the unique branch carrying the nonzero direction $u=(1,0)$.  (The full LICQ version of Theorem~\ref{thm:real radical-branch} is not invoked for the redundant inequality on the upper branch; Proposition~\ref{prop:irredundant} justifies deleting it there.)  On the upper branch
\[
 T_S^2(0;u)=\{w:w_2=2\},
\qquad
 Q_{f_c}(0;u,w)=2(1-c)>0.
\]
This example isolates the optimization content of the branch recovery procedure: it is not merely a different representation of the same test; it prevents a false negative produced by a degenerate polynomial presentation.  The example also shows why, when an active inequality is redundant with a branch equation, one should first simplify the branchwise description rather than impose LICQ on a redundant family.
\end{example}

It is useful to distinguish the present algebraic recovery step from the established theory for switching and disjunctive systems.
\begin{remark}\label{rem:mpsc}
Products of smooth factors, such as $G_1(x)G_2(x)=0$, are the basic algebraic mechanism behind mathematical programs with switching constraints.  For that structured class, explicit directional tangent, normal, outer/inner second order tangent, parabolic derivability, and second order optimality formulas have recently been developed in \cite{ChenEtAl2026}; general disjunctive systems are treated in \cite{BenkoEtAl2023}.  These formulas address a prescribed switching or disjunctive structure.  Theorem~\ref{thm:real radical-branch} addresses a different input: starting from an arbitrary polynomial equality ideal, it recovers smooth real branches through local real radical decomposition and then produces the branchwise parabolic models.  Switching systems are thus a structured special case, whereas the algebraic branch recovery step applies without a prescribed disjunctive representation.
\end{remark}

A separate issue concerns the use of curve selection in quadratic growth arguments.
\begin{remark}\label{rem:puiseux-gap}
Semialgebraic curve selection alone does not justify a parabolic expansion of every selected arc.  Puiseux exponents strictly between $1$ and $2$ may occur.  For the cusp $y^2=x^3$, for instance, a branch has leading behavior $(t,t^{3/2})$ under a first order parametrization.  Such an arc has no finite correction $w$ in an expansion of the form
\[
 \bar x+tu+\tfrac12t^2w+o(t^2).
\]
Thus ordinary parabolic tangent sets do not by themselves control all sequences that enter a quadratic growth argument.  Additional second order derivability or an asymptotic tangent construction is needed in that setting; see \cite{GfrererYeZhou2022,BenkoEtAl2023,OuyangYeZhang2025}.
\end{remark}

\section*{Conclusion}
The results separate two issues that are easily conflated in singular polynomial optimization. Taylor expansion always gives $T_S^2(\bar x;u)\subseteq\cL_S^2(\bar x;u)$, while equality is a geometric realization question. In smooth systems it can be proved by an explicit implicit function construction, including the zero second order inequality case. At singular polynomial representations, the local real radical decomposition can do more than an a priori stratification: when its real minimal prime branches are smooth, branch selection and the exact parabolic tangent set are recovered by finite algebraic tests on Jacobians, gradients, and Hessians. A finite semialgebraic stratification remains an intrinsic geometric fallback, but is no longer required as input in this smooth branch singular class. The real radical hypothesis is essential for real local geometry, the branch formula is invariant under regular generator changes, and counterexamples delimit the roles of branch smoothness and LICQ. For optimization, the exact branch models yield finite linear programs in the parabolic correction and equivalent branchwise multiplier certificates; the false rejection example shows that this preprocessing can change the validity of a second order test, not merely its presentation. Intrinsically singular branches can be incorporated through a plug-in principle once an exact higher order model is supplied by an independent parabolic derivability or $p$-regularity analysis.


\begin{thebibliography}{99}

\bibitem{BenkoEtAl2023}
M. Benko, H. Gfrerer, J. J. Ye, J. Zhang, and J. Zhou, Second-order optimality conditions for general nonconvex optimization problems and variational analysis of disjunctive systems, \emph{SIAM Journal on Optimization} \textbf{33} (2023), 2625--2653. \url{https://doi.org/10.1137/22M1484742}.

\bibitem{BenTal1980}
A. Ben-Tal, Second-order and related extremality conditions in nonlinear programming, \emph{Journal of Optimization Theory and Applications} \textbf{31} (1980), 143--165. \url{https://doi.org/10.1007/BF00934107}.

\bibitem{BenTalZowe1982}
A. Ben-Tal and J. Zowe, A unified theory of first and second order conditions for extremum problems in topological vector spaces, \emph{Mathematical Programming Study} \textbf{19} (1982), 39--76. \url{https://doi.org/10.1007/BFb0120982}.

\bibitem{BochnakCosteRoy1998}
J. Bochnak, M. Coste, and M.-F. Roy, \emph{Real Algebraic Geometry}, Ergebnisse der Mathematik und ihrer Grenzgebiete, vol. 36, Springer, Berlin, 1998. \url{https://doi.org/10.1007/978-3-662-03718-8}.

\bibitem{BCS1998}
J. F. Bonnans, R. Cominetti, and A. Shapiro, Sensitivity analysis of optimization problems under second order regular constraints, \emph{Mathematics of Operations Research} \textbf{23} (1998), 806--831. \url{https://doi.org/10.1287/moor.23.4.806}.

\bibitem{BCS1999}
J. F. Bonnans, R. Cominetti, and A. Shapiro, Second order optimality conditions based on parabolic second order tangent sets, \emph{SIAM Journal on Optimization} \textbf{9} (1999), 466--492. \url{https://doi.org/10.1137/S1052623496306760}.

\bibitem{BonnansShapiro2000}
J. F. Bonnans and A. Shapiro, \emph{Perturbation Analysis of Optimization Problems}, Springer, New York, 2000. \url{https://doi.org/10.1007/978-1-4612-1394-9}.

\bibitem{ChenEtAl2026}
J. Chen, L. Liu, Y. Dai, and E. K\"obis, Directionally variational analysis and second-order optimality conditions for mathematical programs with switching constraints, \emph{Journal of Optimization Theory and Applications} \textbf{208} (2026), Article 22, 45 pp. \url{https://doi.org/10.1007/s10957-025-02849-4}.

\bibitem{Coste2000}
M. Coste, \emph{An Introduction to Semialgebraic Geometry}, Dip. Mat. Univ. Pisa, Dottorato di Ricerca in Matematica, Istituti Editoriali e Poligrafici Internazionali, Pisa, 2000.

\bibitem{FrankowskaOsmolovskii2024}
H. Frankowska and N. P. Osmolovskii, Second order optimality conditions for minimization on a general set. Part 1: Applications to mathematical programming, \emph{Journal of Mathematical Analysis and Applications} \textbf{529} (2024), Article 127384. \url{https://doi.org/10.1016/j.jmaa.2023.127384}.

\bibitem{GfrererYeZhou2022}
H. Gfrerer, J. J. Ye, and J. Zhou, Second-order optimality conditions for nonconvex set-constrained optimization problems, \emph{Mathematics of Operations Research} \textbf{47} (2022), 2344--2365. \url{https://doi.org/10.1287/moor.2021.1211}.

\bibitem{Gibson1979}
C. G. Gibson, K. Wirthmueller, A. A. du Plessis, and E. J. N. Looijenga, \emph{Topological Stability of Smooth Mappings}, Lecture Notes in Mathematics, vol. 552, Springer, Berlin, 1976. \url{https://doi.org/10.1007/BFb0095244}.

\bibitem{Ioffe1991}
A. D. Ioffe, Variational analysis of a composite function: A formula for the lower second-order epi-derivative, \emph{Journal of Mathematical Analysis and Applications} \textbf{160} (1991), 379--405. \url{https://doi.org/10.1016/0022-247X(91)90312-N}.

\bibitem{Lasserre2001}
J. B. Lasserre, Global optimization with polynomials and the problem of moments, \emph{SIAM Journal on Optimization} \textbf{11} (2001), 796--817. \url{https://doi.org/10.1137/S1052623400366802}.

\bibitem{Lojasiewicz1965}
S. Lojasiewicz, \emph{Ensembles semi-analytiques}, IHES notes, 1965.

\bibitem{MMS2021}
A. Mohammadi, B. S. Mordukhovich, and M. E. Sarabi, Parabolic regularity in geometric variational analysis, \emph{Transactions of the American Mathematical Society} \textbf{374} (2021), 1711--1763. \url{https://doi.org/10.1090/tran/8253}.

\bibitem{MohammadiSarabi2020}
A. Mohammadi and M. E. Sarabi, Twice epi-differentiability of extended-real-valued functions with applications in composite optimization, \emph{SIAM Journal on Optimization} \textbf{30} (2020), 2379--2409. \url{https://doi.org/10.1137/19M1300066}.

\bibitem{MohammadiSarabi2025}
A. Mohammadi and M. E. Sarabi, Parabolic regularity of spectral functions, \emph{Mathematics of Operations Research} \textbf{50} (2025), 2017--2046. \url{https://doi.org/10.1287/moor.2023.0010}.

\bibitem{Mordukhovich2006}
B. S. Mordukhovich, \emph{Variational Analysis and Generalized Differentiation I: Basic Theory}, Springer, Berlin, 2006. \url{https://doi.org/10.1007/3-540-31247-1}.

\bibitem{Mordukhovich2018}
B. S. Mordukhovich, \emph{Variational Analysis and Applications}, Springer, Cham, 2018. \url{https://doi.org/10.1007/978-3-319-92775-6}.

\bibitem{NieDemmelSturmfels2006}
J. Nie, J. Demmel, and B. Sturmfels, Minimizing polynomials via sum of squares over the gradient ideal, \emph{Mathematical Programming} \textbf{106} (2006), 587--606. \url{https://doi.org/10.1007/s10107-005-0672-6}.

\bibitem{OuyangYeZhang2025}
W. Ouyang, J. J. Ye, and B. Zhang, New second-order optimality conditions for directional optimality of a general set-constrained optimization problem, \emph{SIAM Journal on Optimization} \textbf{35} (2025), 1274--1299. \url{https://doi.org/10.1137/24M1657493}.

\bibitem{PoliquinRockafellar1996}
R. A. Poliquin and R. T. Rockafellar, Prox-regular functions in variational analysis, \emph{Transactions of the American Mathematical Society} \textbf{348} (1996), 1805--1838. \url{https://doi.org/10.1090/S0002-9947-96-01544-9}.

\bibitem{Ramos2026}
A. Ramos, Second-order optimality conditions for degenerate optimization problems and higher-order tangents sets, \emph{Applied Mathematics \& Optimization} \textbf{94} (2026), Article 5. \url{https://doi.org/10.1007/s00245-026-10466-4}.

\bibitem{Rockafellar1988}
R. T. Rockafellar, First- and second-order epi-differentiability in nonlinear programming, \emph{Transactions of the American Mathematical Society} \textbf{307} (1988), 75--108. \url{https://doi.org/10.2307/2000752}.

\bibitem{Rockafellar1989}
R. T. Rockafellar, Second-order optimality conditions in nonlinear programming obtained by way of epi-derivatives, \emph{Mathematics of Operations Research} \textbf{14} (1989), 462--484. \url{https://doi.org/10.1287/moor.14.3.462}.

\bibitem{RockafellarWets1998}
R. T. Rockafellar and R. J.-B. Wets, \emph{Variational Analysis}, Springer, Berlin, 1998. \url{https://doi.org/10.1007/978-3-642-02431-3}.

\bibitem{Shapiro1997}
A. Shapiro, First and second order analysis of nonlinear semidefinite programs, \emph{Mathematical Programming} \textbf{77} (1997), 301--320. \url{https://doi.org/10.1007/BF02614439}.

\bibitem{vanDenDries1998}
L. van den Dries, \emph{Tame Topology and O-minimal Structures}, Cambridge University Press, 1998. \url{https://doi.org/10.1017/CBO9780511525919}.

\bibitem{Whitney1965}
H. Whitney, Tangents to an analytic variety, \emph{Annals of Mathematics} \textbf{81} (1965), 496--549. \url{https://doi.org/10.2307/1970400}.

\end{thebibliography}
\end{document}